%% file: Arxiv_5_9.tex
\documentclass[12pt]{article}

\usepackage{mathrsfs}

\usepackage{amsmath,amsthm,amssymb}
\usepackage{graphicx}
\usepackage{vector}

\font\tencmmib=cmmib10 \skewchar\tencmmib '60
\newfam\cmmibfam
\textfont\cmmibfam=\tencmmib

\def\lessim{\ \lower4pt\hbox{$
\buildrel{\displaystyle <}\over\sim$}\ }
\def\gessim{\ \lower4pt\hbox{$\buildrel{\displaystyle >}
\over\sim$}\ }

\def\Var{\mathop{\rm Var}\nolimits}
\def\Cov{\mathop{\rm Cov}\nolimits}

\def\d{\mathrm{d}}

\def\<{\langle}
\def\>{\rangle}

\newcommand{\N}{\ensuremath{\mathbb{N}}}
\newcommand{\R}{\ensuremath{\mathbb{R}}}

\newcommand{\e}{\ensuremath{\mathbb{E}}}
\newcommand{\Pro}{\ensuremath{\mathbb{P}}}
\newcommand{\indi}{\ensuremath{\boldsymbol 1}}

\newcommand{\Crt}{\mathop{\mathrm{Crt}}\nolimits}
\usepackage{vector}
\font\tencmmib=cmmib10 \skewchar\tencmmib '60
\newfam\cmmibfam
\textfont\cmmibfam=\tencmmib

\def\lessim{\ \lower4pt\hbox{$
\buildrel{\displaystyle <}\over\sim$}\ }
\def\gessim{\ \lower4pt\hbox{$\buildrel{\displaystyle >}
\over\sim$}\ }

\newcommand{\p}{\mathbb{P}}

\newtheorem{lemma}{\bf Lemma}

\newtheorem{theorem}{\bf Theorem}
\newtheorem{corollary}{\bf Corollary}
\newtheorem{remark}{\bf Remark}

\newenvironment{Proof of lemma}{\noindent{\bf Proof of Lemma}}{\hfill$\Box$\newline}
\newenvironment{Proof of theorem}{\noindent{\bf Proof of Theorem}}{\hfill{\footnotesize${\square}$}\newline}
\newenvironment{Proof of theorems}{\noindent{\bf Proof of Theorems}}{\hfill$\Box$\newline}
\newenvironment{Proof of proposition}{\noindent{\bf Proof of Proposition}}{\hfill$\Box$\newline}
\newenvironment{Proof of propositions}{\noindent{\bf Proof of Propositions}}{\hfill$\Box$\newline}
\newenvironment{Proof of exercise}{\noindent{\it Proof of Exercise:}}{\hfill$\Box$}
\input invlat.tex

\begin{document}

\title{Free energy and complexity of spherical bipartite models}

\author{Antonio Auffinger  \thanks{auffing@math.uchicago.edu} \\ \small{University of Chicago}\and Wei-Kuo Chen \thanks{wkchen@math.uchicago.edu} \\ \small{University of Chicago} }
\maketitle

\begin{abstract}
We investigate both free energy and complexity of the spherical bipartite spin glass model. We first prove a variational formula in high temperature for the limiting free energy based on the well-known Crisanti-Sommers representation of the mixed $p$-spin spherical model. Next, we show that the mean number of local minima at low levels of energy is exponentially large in the size of the system and we derive a bound on the location of the ground state energy.
\end{abstract}

\section{Introduction and main results}

In recent years, spin glass theory has been successfully extended beyond the area of mathematics and mathematical physics, for instance, to the study of social and neural networks \cite{Agliari, Barra0, Dodds, ParisiNet} or theoretical biology and computer science \cite{Amit, Barra2, Barra1, Martelli}. 
Motivated by these fields of research, bipartite spin systems were introduced to model collective properties of two (or more) interactive large groups of individuals. Although very attractive for their realm of applications, bipartite spin glasses are believed to be very difficult to analyze even non-rigorously. For instance, a Parisi-like theory for these models is still under construction within the physics community, and the models lack many of the predictions that other mean field spin glasses share, e.g., the Sherrington-Kirkpatrick (SK) model. A recent account of the predictions can be found in  \cite{Guerra1} and the references therein.

\smallskip

In this paper, we start a rigorous treatment on one class of disordered bipartite spin systems, known as the mixed $(p,q)$-spin spherical bipartite models. Our results concern about the behaviors of the free energy and the extreme values of the Hamiltonian. First, we obtain a variational formula for the limiting free energy at very high temperature in terms of the well-known Crisanti-Sommers formula for the mixed $p$-spin spherical model \cite{Tal07s, Chen}. Surprisingly, our expression (see Theorem \ref{thm1}) is quite different from what has been expected in the bipartite SK model with Ising spin, where the limiting free energy at high temperature is conjectured  to be a minimax formula proposed in \cite{Guerra1, Barra4}.  Second, we provide estimates for the complexity of local minima at any level of energy. We give bounds on the mean number of local minima of the Hamiltonian whose values are below any threshold. These estimates lead to a bound on the location of the ground state energy and establish that the system has a positive complexity. In the physics literature, models with positive complexity are widely known to show glass to spin-glass phase transitions for instance. However, rigorous results about the complexity are scarce.  We believe that our results are the first evidence that the bipartite models share the same phenomena.

\smallskip

We now introduce the spherical bipartite model as follows. For a positive integer $M$, denote by $S^M$ the $(M-1)$-sphere with radius $M^{1/2}$ and by  $\lambda_M$ the uniform probability measure on $S^{M}.$ For any $x=(x_1,\ldots,x_M)$ and $y=(y_1,\ldots,y_M)$ in $\mathbb{R}^M$, set $$R_M(x,y)=\frac{1}{M}\sum_{i=1}^Mx_iy_i.
$$
Let $N, N_1, N_2$ be positive integers that satisfy $N=N_1+N_2$ and
\begin{equation}\label{def:assumption1}
\left|\frac{N_1}{N}-\gamma\right| =o(N)
\end{equation}
for some $\gamma\in(0,1).$
 For any $p,q\geq 1,$ we set the pure $(p,q)$-spin bipartite Hamiltonian as
\begin{equation*}
H_{N,p,q}(u,v) = \sum_{1 \leq i_1, \ldots, i_p \leq N_1} \sum_{1 \leq j_1, \ldots, j_q \leq N_2} g_{i_1,\ldots i_p, j_1, \ldots, j_q}u_{i_1}\ldots u_{i_p} v_{j_1}\ldots v_{j_q}
\end{equation*}
for $u = (u_1, \ldots, u_{N_1}) \in S^{N_1}$ and $v = (v_1, \ldots, v_{N_2}) \in S^{N_2}$, where  all $g$ variables are centered i.i.d. Gaussian random variables with variance $N/(N_1^{p} N_2^{q})$. A computation reveals that $$\e H_{N,p,q}(u,v)  H_{N,p,q}(u',v') = N(R_{N_1}(u,u'))^{p}(R_{N_2}(v,v'))^{q}.$$
The Hamiltonian of the mixed-spin bipartite spherical model is defined as
\begin{equation*}
\label{eq-7}
H_{N}(u,v)+ h_1 \sum_{i=1}^{N_1} u_i + h_2 \sum_{i=2}^{N_2} v_i,
\end{equation*}
where $H_N$ is the linear combination of $H_{N,p,q},$
$$
H_N(u,v)=\sum_{p,q \geq 1} \beta_{p,q} H_{N,p,q}(u,v).
$$
Here $h_1,h_2\in\mathbb{R}$ denote the strength of the external fields and the nonnegative sequence $(\beta_{p,q})_{p,q\geq 1}$ is usually called the (inverse) temperature parameter that satisfies $\beta_{p,q}>0$ for at least one pair $(p,q)$ and decreases fast enough, for instance, $\sum_{p,q=1}^\infty 2^{p+q}\beta_{p,q}^2<\infty$. This assumption ensures the well-definedness of the model and 
a direct computation gives 
$$
\e H_{N}(u,v)H_N(u',v')=N\xi(R_{N_1}(u,u'), R_{N_2}(v,v'))
$$
for
 \begin{align}\label{caraca}
 \xi(x,y):=\sum_{p,q\geq 1}\beta_{p,q}^2x^py^q,\,\,\forall x,y\in[0,1].
 \end{align}

\subsection{Free energy}

We first devote our attention to understanding the bipartite spherical model under the assumptions that the temperature parameter $\xi(1,1)$ and the external fields $h_1,h_2$ are sufficiently small. Define the Gibbs measure of the bipartite spherical model on $S^{N_1}\times S^{N_2}$ by
$$
dG_N(u,v)=\frac{1}{Z_N}\exp\left(H_N(u,v)+h_1\sum_{i=1}^{N_1}u_i+h_2\sum_{i=1}^{N_2}v_i\right)d(\lambda_{N_1}\times \lambda_{N_2})(u,v),
$$
where the normalizing factor $Z_N$ is called the partition function. For a fixed realization $g,$ we denote by $(u^1,v^2)$ and $(u^2,v^2)$ i.i.d. sampled configurations (replicas) with respect to $G_N$. The overlaps within the first and second parties are defined respectively as
\begin{align*}
R_{1,2}^1&=R_{N_1}(u^1,u^2)=\frac{1}{N_1}\sum_{i=1}^{N_1}u_i^1 u_i^{2}\,\,\mbox{and}\,\,
R_{1,2}^2=R_{N_2}(v^{1},v^{2})=\frac{1}{N_2}\sum_{i=1}^{N_2}v_i^1 v_i^{2}.
\end{align*}
Finally, we set the free energy as 
\begin{align*}
F_N=\frac{1}{N}\e\log Z_N.
\end{align*}

\smallskip

First, we will show that the limiting free energy can be computed through an analogy of the well-known Cransti-Sommers formula in the mixed $p$-spin spherical model. Second, we will show that the overlap within each party is concentrated around a constant with a qualitative error estimate. For notational convenience, throughout this section we will use $M_0,\lambda_0>0$ to stand for very small quantities and $L$ is a universal constant independent of everything. These quantities might be different form each occurrence. Also, as we will only be concerned with very small $\xi(1,1)=\sum_{p,q\geq 1} \beta_{p,q}^2$, we will use the clear fact without mentioning at each time that the partial derivatives of $\xi$ up to the second order are uniformly small as well.

\smallskip

We now state our main results. Define $L:[0,1)^2\rightarrow\mathbb{R}$ by
\begin{align*}
P(a,b)&=\frac{\gamma}{2}\left(h_1^2(1-a)+\frac{a}{1-a}+\log(1-a)+\xi(1,1)-\xi(a,b)\right)\\
&+\frac{1-\gamma}{2}\left(h_2^2(1-b)+\frac{b}{1-b}+\log(1-b)+\xi(1,1)-\xi(a,b)\right).
\end{align*}

\begin{theorem}\label{thm1}
There exists a small $M_0>0$ such that the following statements hold whenever $\xi(1,1),|h_1|,|h_2|<M_0$.
\begin{itemize}
\item[$(i)$] We have that
\begin{align}
\begin{split}\label{eq0}
\lim_{N\rightarrow\infty}F_N&=\min_{a,b\in[0,1)}P(a,b).
\end{split}
\end{align}
Here, the minimum is attained by $(a_0,b_0)\in [0,1)^2$ that is the unique solution to the following system of equations
\begin{align*}
h_1^2+\frac{\partial_x\xi(a_0,b_0)}{\gamma}&=\frac{a_0}{(1-a_0)^2},\,\,h_2^2+\frac{\partial_y\xi(a_0,b_0)}{1-\gamma}=\frac{b_0}{(1-b_0)^2}.
\end{align*}
\item[$(ii)$] Let $\left<\cdot\right>$ denote the expectation with respect to the product measure $G_N\times G_N.$ For the overlaps,
\begin{align}
\begin{split}\label{eq-1}
\e\left<(R_{1,2}^1-a_0)^2\right>,\,\,\e\left<(R_{1,2}^2-b_0)^2\right>&\leq \frac{L\log^2N}{N},
\end{split}
\end{align}
where $L$ is a universal constant independent of everything.
\end{itemize}
\end{theorem} 

Theorem \ref{thm1} is the high temperature results ($\xi(1,1)$ very small). In \cite{Guerra1, Barra4}, the limiting free energy of the bipartite SK model in the Ising spin case is conjectured to be a variational formula of minimax type. The formula \eqref{eq0} shows that the Cransti-Sommers formula in our bipartite model does not share the minimax characterization. In the case of $h_1=h_2=0$, we clearly see that the overlap within each party is concentrated at $0$ sharing the same behavior as the Haar measure on the product of the two spheres.

\smallskip

The proof of Theorem \ref{thm1} is mainly motivated by a beautiful argument of Latala as presented in \cite{Pan} and Chapter 1 of \cite{Tal11}. His argument originated from the study of the SK model in Ising spin case in the high temperature regime. We found that it works equally well in the bipartite model. To extend Theorem \ref{thm1}$(i)$ to the low temperature case, i.e., $\xi(1,1)$ is very large, a reasonable approach is to adapt the treatments in the study of the spherical SK model \cite{Tal07s}, where the Guerra replica symmetry breaking bound (RSB) has played an essential role in the computation of the limiting free energy. However, an equivalent formulation of the Guerra RSB bound does not seem to hold for any clear reason in the setting of the bipartite model, see Remark \ref{rmk1}. Some new ideas are needed.

\subsection{Complexity of local minima}

Our goal in this section is to compute the complexity of local minima of the model, i.e., the mean number of local minima below a certain level of the function $H_{N}$.  This computation has been carried out in the mixed $p$-spin spherical model in \cite{AB, ABC} and also in other examples of disordered systems (see \cite{FyoLN} and the references therein). Note that as $\xi$ has continuous partial derivatives up to fourth order on $[0,1]^2,$ this guarantees the differentiability of $H_N$ for almost surely all sampled realization $g$ so that we can study the critical points of $H_N$. Since in this section the temperature does not play a role, we will now assume that $\xi(1,1) =1.$
In other words, the Gaussian process $H_N$ has constant variance $N$.

\smallskip

For any $t \in \R$ and $0\leq k \leq N-1$, the complexity function is defined as the random number
$\Crt_{N,k}(t)$ of critical points of index $k$ of the
function $H_{N }$ below level $Nt$. More precisely,
\begin{equation}
  \label{defCritical}
  \Crt_{N,k}(t) = 
  \sum_{(u,v): \nabla H_{N}(u,v) = 0 } 
  \indi\{ H_{N}(u,v) \leq Nt\} \indi\{
    i(\nabla^2 H_{N}(u,v)) = k\},
\end{equation}
where  $\indi \{ E \} = \indi_{\{ E \}}$ denotes the indicator function of the event $E$. In particular, $\Crt_{N,0}(t)$ denotes the number of local minima below level $Nt$. Here $\nabla$ and $\nabla^2$ are the gradient and the Hessian restricted to $S^{N_1} \times S^{N_2}$ and  $i(\nabla^2 H_{N}(u,v))$ is the number of negative eigenvalues of the Hessian $\nabla^2 H_{N}$, called the index of the Hessian at $(u,v)$. Our main result is  lower and upper bounds for the mean complexity of local minima.

\begin{theorem}\label{upperlower} Assume that $\xi(x,y) \neq x y$. There exist continuous functions  $J, K: \mathbb R \rightarrow \mathbb R$ such that
\begin{equation}\label{salt}
\begin{split}
J(t)\leq \lim_{N \rightarrow \infty} \frac{1}{N} \log \e \Crt_{N,0}(t) \leq K(t).
\end{split}
\end{equation}
Furthermore, $(\text{Im } J) \cap (0,\infty) \neq \emptyset$ and $\lim_{t\rightarrow -\infty} K(t) = -\infty$.
\end{theorem}

The functions $J$ and $K$ above are derived from the analysis of the modified Hamiltonian  obtained by coupling two independent mixed $p$-spin spherical models together, where explicit formulas are available. These functions are defined in  Section \ref{sec:s5} (see \eqref{eq:Psi} and  \eqref{eq:rio}) through certain variational principles. The exact expressions of $J$ and $K$ do not play an important role here. However, these bounds provide essential consequences. On the one hand, the upper bound in \eqref{salt} gives a straightforward bound on the location of the minimum of the Gaussian process $H_N$ as the function $K$ takes both negative and positive values.  Indeed, a simple application of the Markov inequality combined with the theorem above leads to
\begin{corollary} Let $m_0$ be the smallest zero of the function $K(t)$. For any $\epsilon>0$ there exist positive constants $C_1, C_2$ such that
\begin{equation*}
\mathbb{P} (\min H_N \leq N(m_0 -\epsilon))  \leq C_1 \exp(-C_2N).
\end{equation*}
\end{corollary}
\begin{proof}
Just note that $\min H_N \leq  N t$ if and only if $\Crt_{N,0}(t) \geq 1$.
\end{proof}
On the other hand, the lower bound in \eqref{salt} establishes what is known in the physics community as the presence of a positive complexity. Precisely, as the function $J$ takes positive values, Theorem \ref{upperlower} shows an exponentially large mean number of local minima of the energy (even at deep values of $H_N$). Although the result here is about the mean number, it is predicted, at least in the pure case, $\xi(x,y) = x^p y^q$, that $\Crt_{N,0}(t)$ concentrates around its mean. Despite the physicists believe, this self-averaging property is, as far as we know, open in {\it all} disordered models of spin glasses. 

The rest of the paper is organized as follows. In the next section, we provide a proof of Theorem \ref{thm1}. In Section \ref{sec:s4}, we study the complexity of the modified Hamiltonian mentioned above. The results of Section \ref{sec:s4} are of independent interest and are used in the proof of Theorem \ref{upperlower} in Section \ref{sec:s5}. 

\section{Proof of Theorem \ref{thm1}}
For arbitrary $a,b\in [0,1],$ consider the interpolated Hamiltonian $H_t$ for $0\leq t\leq 1,$
\begin{align*}
\begin{split}
H_t(u,v)&=\sqrt{t}H_N(u,v)+\sum_{i=1}^{N_1}u_i\left(\sqrt{1-t}y_i\sqrt{\gamma_N^{-1}\partial_x\xi(a,b)}+h_1\right)\\
&\qquad\qquad\qquad\quad+\sum_{j=1}^{N_2}v_j\left(\sqrt{1-t}z_j\sqrt{(1-\gamma_N)^{-1}\partial_y\xi(a,b)}+h_2\right),
\end{split}
\end{align*}
where $y_1,\ldots,y_{N_1},z_1,\ldots,z_{N_2}$ are i.i.d. standard Gaussian and $\gamma_{N} := {N_1}/{N}$. Let $G_{t}$ and $Z_t$ be the Gibbs measure and partition function associated to the Hamiltonian $H_t$, respectively. We denote by $(u^{\ell},v^{\ell})_{\ell\geq 1}$ a sequence of i.i.d. sampled configurations (replicas) with respect to $G_t$ and for any bounded measurable function $f$ depending on $(u^{\ell},v^{\ell})_{1\leq \ell\leq n}$ for some $n\geq 1$, we use $\left<f\right>_t$ to stand for its Gibbs expectation, i.e.,
\begin{align*}
\left<f\right>_t=\int_{(S^{N_1}\times S^{N_2})^n}f(u^1,v^1,\ldots,u^n,v^n)dG_t(u^1,v^1)\times\cdots\times G_t(u^n,v^n).
\end{align*}
Define overlaps within the first and second parties as
\begin{align*}
R_{\ell,\ell'}^1&=R_{N_1}(u^{\ell},u^{\ell'})=\frac{1}{N_1}\sum_{i=1}^{N_1}u_i^\ell u_i^{\ell'}
\,\,\mbox{and}\,\,
R_{\ell,\ell'}^2=R_{N_2}(v^{\ell},v^{\ell'})=\frac{1}{N_2}\sum_{i=1}^{N_2}v_i^\ell v_i^{\ell'}.
\end{align*}
Set the interpolated free energy
\begin{align*}
\phi(t)&=\frac{1}{N}\e \log Z_t.
\end{align*}

First we want to control the derivative of $\phi(t).$ Using the Gaussian integration by parts, 
\begin{align*}
\phi'(t)&=\frac{1}{2N}\e\left<\frac{H_N(u,v)}{\sqrt{t}}-\frac{\sum_{i=1}^{N_1}u_iy_i\sqrt{\gamma_N^{-1}\partial_x\xi(a,b)}+\sum_{j=1}^{N_2}v_iz_i\sqrt{(1-\gamma_N)^{-1}\partial_y\xi(a,b)}}{\sqrt{1-t}}\right>_t\\
&=-\frac{1}{2}\e\left<\triangle(R_{1,2}^1,R_{1,2}^2)\right>_t+\frac{1}{2}\triangle(1,1),
\end{align*}
where
\begin{align*}
\triangle(x,y)&:=\xi(x,y)-(x\partial_x\xi(a,b)+y\partial_y\xi(a,b))+\theta(a,b)\\
\theta(a,b)&:=a\partial_x\xi(a,b)+b\partial_y\xi(a,b)-\xi(a,b).
\end{align*}

\begin{remark}\label{rmk1}
\rm As the function $\triangle$ may take both positive and negative values, it does not seem so easy to see that the function $\phi'$ takes only one sign. For instance, $\triangle(x,y)=(x-a)(y-b)$ in the simplest case $\xi(x,y)=xy.$ Therefore, it is by no means clear whether we have Guerra's bound, $F_N=\phi(1)\leq \phi(0)$, or not.
\end{remark}

For any $x,y\in[0,1]$, Taylor's formula says that there exist $a',b'\in [0,1]$ such that
\begin{align*}
\xi(x,y)&=\xi(a,b)+\partial_x\xi(a,b)(x-a)+\partial_y\xi(a,b)(y-b)\\
&+\frac{1}{2}\left<\triangledown^2 \xi(a',b')\left[\begin{array}{c}
x-a\\
y-b
\end{array}\right],\left[\begin{array}{c}
x-a\\
y-b
\end{array}\right]\right>\\
&=x\partial_x\xi(a,b)+y\partial_y\xi(a,b)-\theta(a,b)\\
&+\frac{1}{2}\left<\triangledown^2 \xi(a',b')\left[\begin{array}{c}
x-a\\
y-b
\end{array}\right],\left[\begin{array}{c}
x-a\\
y-b
\end{array}\right]\right>.
\end{align*}
Using this and letting $M=\sup_{x,y\in[0,1]}|\triangledown^2 \xi(x,y)|$ and $C(x,y)=\max\{|x-p|,|y-q|\}$ give 
\begin{align}\label{eq2}
|\triangle(x,y)|\leq MC(x,y).
\end{align}
Therefore, to control the derivative of $\Phi(t)$, it suffices to control $\e\left<C(R_{1,2}^1,R_{1,2}^2)\right>_t.$ We will argue by Latala's argument.
For a function $f=f(u^1,v^1,\ldots,u^n,v^n)$, a direct computation using again Gaussian integration by parts gives
\begin{align}
\begin{split}
\label{eq1}
\frac{d}{dt}\e\left<f\right>_t&=N\sum_{1\leq \ell<\ell'\leq n}\e\left<\triangle(R_{\ell,\ell'}^1,R_{\ell,\ell'}^2)f\right>_t-Nn\sum_{\ell\leq n}\e\left<\triangle(R_{\ell,n+1}^1,R_{\ell,n+1}^2)f\right>_t\\
&\,\,+\frac{n(n+1)N}{2}\e\left<\triangle(R_{n+1,n+2}^1,R_{n+1,n+2}^2f)\right>_t.
\end{split}
\end{align}
For any $\ell\neq\ell'$, we define 
$$
C_{\ell,\ell'}=C(R_{\ell,\ell'}^1,R_{\ell,\ell'}^2).
$$
Observe that from H\"{o}lder's inequality and the symmetry between replicas,
\begin{align*}
\e \left<C_{\ell,\ell'}^2\exp \lambda NC_{1,2}^2\right>_t&=\sum_{k=0}^{\infty}\frac{(\lambda N)^{k}}{k!}\e\left<C_{\ell,\ell'}^2C_{1,2}^{2k}\right>_t\\
&\leq \sum_{k=0}^{\infty}\frac{(\lambda N)^{k}}{k!}\left(\e\left<C_{\ell,\ell'}^{2k+2}\right>_t\right)^{1/(k+1)}\cdot\left(\e \left<C_{1,2}^{2k+2}\right>_t\right)^{k/(k+1)}\\
&=\e \left<C_{1,2}^2\exp \lambda NC_{1,2}^2\right>_t.
\end{align*}
Using this inequality, the equation \eqref{eq1} with $n=2$ and $f=\exp \lambda NC_{1,2}^2$ and \eqref{eq2}, we obtain
\begin{align*}
\frac{d}{dt}\e\left<\exp \lambda NC_{1,2}^2\right>_t&\leq 8MN\e\left<C_{1,2}^2\exp \lambda NC_{1,2}^2\right>_t,
\end{align*}
which yields that for $t\leq \lambda/8M,$
\begin{align}\label{eq7}
\frac{d}{dt}\e\left<\exp(\lambda-8tM)NC_{1,2}^2\right>_t&\leq 0.
\end{align}

From now on, for each $N\geq 1,$ we pick $(a,b)=(a_N,a_N)$ such that 
\begin{align*}
a_N&=\e\left<R_{1,2}^1\right>_0,\\
b_N&=\e\left<R_{1,2}^2\right>_0.
\end{align*}
Let us remark that since $(a,b)\mapsto (\e\left<R_{1,2}^1\right>_0,\e\left<R_{1,2}^2\right>_0)$ defines a continuous function on the compact and convex space $[0,1]^2,$ this ensures the existence of such $(a_N,b_N)$ by Brouwer's fixed point theorem. We need the following lemma.

\begin{lemma}
\label{lem1}
There exist small $M_0>0$ and $\lambda_0>0$ such that whenever $\xi(1,1),|h_1|,|h_2|<M_0$, we have
\begin{align}
\label{lem1:eq1}
\e\left<\exp \lambda_0 NC_{1,2}^2\right>_0\leq L,
\end{align}
where $L$ is a constant independent of everything. In addition, the following system of equations has a unique solution $(a_0,b_0)\in [0,1)^2$,
\begin{align}
\begin{split}\label{eq6}
h_1^2+\frac{\partial_x\xi(a_0,b_0)}{\gamma}&=\frac{a_0}{(1-a_0)^2},\,\,
h_2^2+\frac{\partial_y\xi(a_0,b_0)}{1-\gamma}=\frac{b_0}{(1-b_0)^2}
\end{split}
\end{align}
and $(a_0,b_0)$ satisfies 
\begin{align}
\label{eq5}
|a_N-a_0|,|b_N-b_0|\leq\frac{L\log^2 N}{N}.
\end{align}
\end{lemma}

\begin{proof}
Note that $e^{\max(x,y)}\leq e^x+e^y$ for any $x,y$. It suffices to show that 
\begin{align*}
\e\left<\exp \lambda N(R_{1,2}^1-a_N)^2\right>_0,\,\,\e\left<\exp \lambda N(R_{1,2}^2-b_N)^2\right>_0\leq L.
\end{align*} 
Since the configuration space in the bipartite model is $S^{N_1}\times S^{N_2}$, $R_{1,2}^1$ depends only on $u^1,u^2\in S^{N_1}$ and $R_{1,2}^2$ depends only on $v^1,v^2\in S^{N_2}$, we have that
\begin{align*}
\e\left<\exp \lambda N(R_{1,2}^1-a_N)^2\right>_0&=\e\left<\exp \lambda N(R_{1,2}^1-a_N)^2\right>_1^{-},\\
\e\left<\exp \lambda N(R_{1,2}^2-b_N)^2\right>_0&=\e\left<\exp \lambda N(R_{1,2}^2-b_N)^2\right>_2^{-},
\end{align*}
where $\left<\cdot\right>_1^-$ and $\left<\cdot\right>_2^-$ stand for the Gibbs expectations associated respectively to the Hamiltonians 
\begin{align}
\begin{split}
\label{eq12}
&\sum_{i=1}^Nu_i\left(y_i\sqrt{\gamma_N^{-1}\partial_x\xi(a_N,b_N)}+h_1\right),\,\,\forall u\in S^{N_1}\\
&\sum_{j=1}^Nv_j\left(z_j\sqrt{(1-\gamma_N)^{-1}\partial_y\xi(a_N,b_N)}+h_2\right),\,\,\forall v\in S^{N_2}.
\end{split}
\end{align}
For Hamiltonians like \eqref{eq12}, they have been studied thoughtfully in Panchenko \cite{Pan}. In particular, one can find that exactly the same argument with only some minor modifications as pages $1045- 1046$ in \cite{Pan} implies that there exists some $M_0>0$ and $\lambda_0>0$ such that whenever $\xi(1,1),|h_1|,|h_2|<M_0$,
\begin{align*}
\e\left<\exp \lambda_0 N(R_{1,2}^1-a_N)^2\right>_{1}^-\leq L,\\
\e\left<\exp \lambda_0 N(R_{1,2}^2-b_N)^2\right>_{2}^-\leq L,
\end{align*}
which gives \eqref{lem1:eq1}. In addition, the equation $(1.5)$ and Lemma 7 in  \cite{Pan} says that there exist
$a_N'$ and $b_N'$ satisfying
\begin{align*}
h_1^2+\frac{\partial_x\xi(a_N',b_N)}{\gamma_N}&=\frac{a_N'}{(1-a_N')^2},\\
h_2^2+\frac{\partial_y\xi(a_N,b_N')}{1-\gamma_N}&=\frac{b_N'}{(1-b_N')^2}
\end{align*}
such that
\begin{align}\label{eq3}
|a_N-a_N'|,|b_N-b_N'|&\leq \frac{L\log^2 N}{N}.
\end{align}
These amount to say that
\begin{align}
\begin{split}\label{eq4}
\left|h_1^2+\frac{\partial_x\xi(a_N,b_N)}{\gamma_N}-\frac{a_N}{(1-a_N)^2}\right| \leq\frac{L\log^2 N}{N},\\
\left|h_2^2+\frac{\partial_y\xi(a_N,b_N)}{1-\gamma_N}-\frac{b_N'}{(1-b_N)^2}\right| \leq\frac{L\log^2 N}{N}.
\end{split}
\end{align}
Define 
\begin{align*}
\begin{split}
F_1(a,b)&=(1-a)^2\left(h_1^2+\frac{\partial_x\xi (a,b)}{\gamma}\right),\\
F_2(a,b)&=(1-b)^2\left(h_2^2+\frac{\partial_y\xi(a,b)}{1-\gamma}\right),\\
F(a,b)&=\left(F_1(a,b),F_2(a,b)\right)
\end{split}
\end{align*}
for $(a,b)\in [0,1]\times[0,1].$ Since $h_1,h_2,\xi(1,1)$ are all small, $F$ maps into itself and a direct computation shows that $\sup_{a,b\in[0,1]}\|DF(a,b)\|<1$. So $F$ forms a contraction mapping and there exists a unique $(a_0,b_0)\in [0,1]^2$ such that $F(a_0,b_0)=(a_0,b_0).$ Note that here none of $a_0$ and $b_0$ could be equal to $1$. To sum up, we obtain \eqref{eq6}. On the other hand, since $G(a,b):=(F_1(a,b)/(1-a)^2,F_2(a,b)/(1-b)^2)$ for $(a,b)\in [0,1)^2$ satisfies
\begin{align*}
DG(a,b)&=\left[
\begin{array}{cc}
\frac{\partial_{xx}\xi(a,b)}{\gamma}-\frac{1+a}{(1-a)^3}&\frac{\partial_{xy}\xi(a,b)}{\gamma}\\
\frac{\partial_{xy}\xi(a,b)}{1-\gamma}&\frac{\partial_{yy}\xi(a,b)}{1-\gamma}-\frac{1+b}{(1-b)^3}
\end{array}\right]
\end{align*}
and $\xi(1,1)$ is small, this ensures that $DG$ is invertible in a small open neighborhood of  $(a_0,b_0)$. This combining with \eqref{eq4} and the assumption that $|\gamma_N-\gamma|\leq K/N$ gives \eqref{eq5}.
\end{proof}

Recall the definition of $M=\max_{x,y\in[0,1]}|\triangledown^2\xi(x,y)|.$ Using this lemma, we take $\xi(1,1)<M_0$ to be small enough such that $16M<\lambda_0.$ Then for any $0\leq t\leq 1,$ $\lambda_0-8Mt\geq 16M-8Mt\geq 8M$ and consequently, from \eqref{eq7},
\begin{align*}
\e\left<\exp 8MNC_{1,2}^2\right>_t&\leq \e\left<\exp (\lambda_0-8Mt)NC_{1,2}^2\right>_t\\
&\leq\e\left<\exp \lambda_0NC_{1,2}^2\right>_0\\
&\leq L.
\end{align*}
In other words, from $x^{2}\leq 2e^{x^2}$ for all $x\in\mathbb{R}$, we obtain that for all $0\leq t\leq 1,$
\begin{align}
\label{eq-2}
\e\left<C_{1,2}^2\right>_t&\leq \frac{L}{N}
\end{align}
and therefore, 
\begin{align}
\label{eq9}
\left|\phi(1)-\phi(0)-\frac{1}{2}\triangle(1,1)\right|\leq\frac{L}{N}.
\end{align}
Now note that $\phi(1)=N^{-1}\e\log Z_N$ and
\begin{align*}
\phi(0)&=\frac{1}{N}\e\log \int_{S^{N_1}}\exp \sum_{i=1}^{N_1}u_i\left(y_i\sqrt{\gamma_N^{-1}\partial_x\xi(a_N,b_N)}+h_1\right)d\lambda_{N_1}\\
&+\frac{1}{N}\e\log \int_{S^{N_2}}\exp \sum_{j=1}^{N_2}v_j\left(z_i\sqrt{(1-\gamma_N)^{-1}\partial_y\xi(a_N,b_N)}+h_2\right)d\lambda_{N_2}
\end{align*}
Since $\lim_{N\rightarrow\infty}a_N=a_0$, $\lim_{N\rightarrow\infty}b_N=b_0$ and $\lim_{N\rightarrow\infty}\gamma_N=\gamma$, these guarantee that
\begin{align}
\begin{split}
\label{eq11}
\lim_{N\rightarrow\infty}\phi(0)&=\gamma\lim_{N\rightarrow\infty}\frac{1}{N_1}\e\log \int_{S^{N_1}}\exp \sum_{i=1}^{N_1}u_i\left(y_i\sqrt{\gamma^{-1}\partial_x\xi(a_0,b_0)}+h_1\right)d\lambda_{N_1}\\
&+(1-\gamma)\lim_{N\rightarrow\infty}\frac{1}{N_2}\e\log \int_{S^{N_2}}\exp \sum_{j=1}^{N_2}v_j\left(z_i\sqrt{(1-\gamma)^{-1}\partial_y\xi(a_0,b_0)}+h_2\right)d\lambda_{N_2}.
\end{split}
\end{align}
Here the two limits on the right-hand side are known as the simplest cases (the Hamiltonians are defined only by the external field) of the limiting free energy for the spherical SK model and are given respectively by the Crisanti-Sommers formula \cite[Theorem 1.1 and Equation (1.11) therein]{Tal07s},
\begin{align}
\begin{split}
\label{eq10}
\frac{1}{2}\inf_{a\in[0,1]}\left\{h_1^2(1-a)+\frac{a}{1-a}+\log(1-a)+\frac{(1-a)\partial_x\xi(a_0,b_0)}{\gamma}\right\},\\
\frac{1}{2}\inf_{b\in[0,1]}\left\{h_2^2(1-b)+\frac{b}{1-b}+\log(1-b)+\frac{(1-b)\partial_y\xi(a_0,b_0)}{1-\gamma}\right\}.
\end{split}
\end{align}
Here the critical point equations in both cases are given by
\begin{align*}
h_1^2+\frac{\partial_x\xi(a_0,b_0)}{\gamma}&=\frac{a}{(1-a)^2},\\
h_2^2+\frac{\partial_y\xi(a_0,b_0)}{1-\gamma}&=\frac{b}{(1-b)^2}.
\end{align*}
Since $(a_0,b_0)$ also satisfies \eqref{eq6}, it follows $(a,b)=(a_0,b_0).$ This combining with \eqref{eq5}, \eqref{eq9}, \eqref{eq11} and \eqref{eq10} yields 
\begin{align*}
\lim_{N\rightarrow\infty}\frac{1}{N}\e\log Z_N&=\frac{\gamma}{2}\left(h_1^2(1-a_0)+\frac{a_0}{1-a_0}+\log(1-a_0)\right)+\frac{(1-a_0)\partial_x\xi(a_0,b_0)}{2}\\
&+\frac{(1-\gamma)}{2}\left(h_2^2(1-b_0)+\frac{b_0}{1-b_0}+\log(1-b_0)\right)+\frac{(1-b_0)\partial_y\xi(a_0,b_0)}{2}\\
&+\frac{1}{2}\left(\xi(1,1)-\xi(a_0,b_0)-(1-a_0)\partial_x\xi(a_0,b_0)-(1-b_0)\partial_y\xi(a_0,b_0)\right)\\
&=P(a_0,b_0).
\end{align*}
Since 
\begin{align*}
\triangledown P(a,b)&=\frac{1}{2}\left(-\gamma h_1^2+\frac{\gamma a}{(1-a)^2}-\partial_x\xi(a,b),-(1-\gamma) h_2^2+\frac{(1-\gamma)b}{(1-b)^2}-\partial_y\xi(a,b)\right),
\end{align*}
it means that $(a_0,b_0)$ is the only critical point of $P$. On the other hand, since
\begin{align*}
\triangledown^2P(a,b)
&=\frac{1}{2}\left(
\begin{array}{cc}
\frac{\gamma(1+a)}{(1-a)^3}-\partial_{xx}\xi(a,b)&-\partial_{xy}\xi(a,b)\\
-\partial_{xy}\xi(a,b)&\frac{(1-\gamma)(1+b)}{(1-b)^3}-\partial_{yy}\xi(a,b)
\end{array}
\right),
\end{align*}
$a,b\in [0,1)$ and $\xi(1,1)$ is very small, we see that $\triangledown^2P(a,b)$ is positive semi-definite on $[0,1)^2.$ Therefore, $(a_0,b_0)$ attains the infimum of $P$. This gives \eqref{eq0}. As for \eqref{eq-1}, it can be obtained from \eqref{eq3} and \eqref{eq-2}.

\section{Complexity of a coupled Hamiltonian}\label{sec:s4}
As the first step to prove Theorem \ref{upperlower}, we will need to investigate the complexity of a coupled Hamiltonian using two independent pure $p$-spin spherical models. We believe the result here is also of interest by itself. Consider two independent mixed-spin spherical Hamiltonians $(\mathcal{H}_{N_1}^1(u):u\in S^{N_1}\}$ and $(\mathcal{H}_{N_2}^2(v):v\in S^{N_2})$ with mean zero and variance,
\begin{align}\label{eq:56}
\e \mathcal H_{N_1}^1(u) \mathcal{H}_{N_1}^1(u')= N_1 \xi^1(R^1(u,u'))\,\,\mbox{and}\,\,\e \mathcal H_{N_2}^2(v) \mathcal{H}_{N_2}^2(v')= N_2 \xi^2(R^2(v,v')),
\end{align} 
where $\xi^1(x):=\sum_{p\geq 1} \beta_p^2 x^{p}$ and $\xi^2(y):=\sum_{q\geq 1}\beta_q'^2 y^{q}.$ As before, we will assume that the sequences $(\beta_p)_{p\geq 1}$ and $(\beta_q')_{q\geq 1}$ decay sufficiently fast to zero so that $\mathcal{H}_{N_1}^1$ and $\mathcal{H}_{N_2}^2$ are smooth almost surely and we assume $\xi_1(1)=1=\xi_2(1).$ The coupled Hamiltonian we will be considering is defined as 
\begin{align}
\label{split}
\tilde{H}_N(u,v)=\mathcal{H}_{N_1}^1(u)+\mathcal{H}_{N_2}^2(v).
\end{align}
For this coupled Hamiltonian we can precisely derive the asymptotic complexity.  We will use $\Crt_{N,k}(t)$ for the variable defined in \eqref{defCritical} with $H_N$ replaced by $\bar H_N$. Our main result  in this direction is the following.

\begin{theorem}\label{thm:easycomplexity} For any $k \in \N\cup\{0\}$, the following limit exists,
\begin{equation*}
\Upsilon_k(t):=\lim_{N \rightarrow \infty} \frac{1}{N} \log \e \Crt_{N,k}(t)< \infty.
\end{equation*}
\end{theorem}

\begin{proof} For simplicity, we will assume $N_1 = \gamma N$ and $N_2 = (1-\gamma)N$. 
All the arguments below hold under minor changes to \eqref{def:assumption1}. We fix an orthonormal basis $\Omega(u,v)$ of the tangent plane of $S^{N_1} \times S^{N_2}$ at $(u,v)$ such that $\Omega(u,v) = (\Omega_1(u), \Omega_2(v)) =(\Omega_1, \Omega_2)$, where $\Omega_i = \{E^1_i, \ldots, E^{N_i-1}_i \}$ is an orthonormal basis of $S^{N_i}$ at $u$ (or $v$) respectively. We choose $\Omega(u,v)$ continuously on $(u,v)$. In view of \eqref{split}, one can write 
\begin{equation}\label{GradientDecomposition}
 \nabla \tilde H_N(u,v) = (\nabla \mathcal H_{N_1}^1(u), \nabla \mathcal H_{N_2}^2(v)) 
 \end{equation}
and 
\begin{equation} \label{HessianDecomposition}
\nabla^2 \tilde H_N(u,v) =\left( \begin{array}{cc}
\nabla^2 \mathcal H_{N_1}^{1}(u) & 0  \\
0 &  \nabla^2 \mathcal H_{N_2}^{2}(v) \end{array} \right).
\end{equation}
We now write the event $ \{ i(\nabla^2 \tilde H_N(u,v))=k \}$  as the following disjoint union
\begin{equation*}
\left\{ i(\nabla^2 \tilde H_N(u,v))=k \right\} = \bigcup_{l=0}^k \left\{ i(\nabla^2 \mathcal H_{N_1}^{1}(u))=l,  i(\nabla^2 \mathcal H_{N_2}^{2}(v))=k-l \right\} =  \bigcup_{l=0}^k (E_l^1 \cap E_{k-1}^2),
\end{equation*}
which leads to 
\begin{equation*}
\e  \Crt_{N,k}(t) = \sum_{l=0}^{k}  \e\Crt_{N,k}(t) \indi_{E_l^1} \indi_{E_{k-l}^2}.
\end{equation*}
Conditioning on the value of $\mathcal H_{N_1}^1$, each term of the sum above can be written as 
\begin{equation*}
\begin{split}
&\frac{\sqrt{N_1}}{\sqrt{2\pi}}\int_{-\infty}^{\infty} \e \big[ \# \{ (u,v) : \nabla \mathcal H_{N_1}^{1}(u) = 0,\,\,\nabla \mathcal H_{N_2}^{2}(v) =0,\\
&\qquad \mathcal H_{N_2}^2(v) \leq Nt - N_1x, E_l^1, E_{k-l}^2 \} \big | \mathcal H_{N_1}^{1}(u) =N_1x \big]e^{- \frac{N_1x^2}{2}} dx \\
&=\frac{\sqrt{N_1}}{\sqrt{2\pi}} \int_{-\infty}^{\infty} \e [\Crt_{N_2,k-l}^{\mathcal H_{N_2}^2}((1-\gamma)^{-1}(t-\gamma x))] \\
&\qquad\cdot\e \big[ \# \{ u: \nabla \mathcal H_{N_1}^1(u) =0, E_l^1 \} \big | \mathcal H_{N_1}^1(u) = N_1 x \big]  e^{- \frac{N_1x^2}{2}} dx
\end{split}
\end{equation*}
Now from \cite[Theorem 1.1]{AB} (see also \cite[Theorem 2.5]{ABC} for the pure case), there exist functions $\Theta_l$ and $\Lambda_k$ from $\mathbb R$ to $ \mathbb R \cup \{ -\infty \}$, both bounded above and smooth when finite, such that
\begin{align*}
 \e [\Crt_{N_2,k-l}^{\mathcal H_{N_2}^2}((1-\gamma)^{-1}(t-\gamma x))] &\sim \exp(N_2\Theta_{k-l}((1-\gamma)^{-1}(t-\gamma x)) \nonumber \\
 \e \big[ \# \{ u: \nabla \mathcal H_{N_1}^1(u) =0, E_l^1 \} \big | \mathcal H_{N_1}^1(u) = N_1 x \big] &\sim \exp(N_1\Lambda_l(x)) 
 \end{align*}
where $a_n(x) \sim b_n(x)$ means $\lim n^{-1}\log (a_n(x)/b_n(x))=0$ uniformly in $x$.   This allows us to replace the above limits in \eqref{split} to obtain after an application of Laplace's Method the expression
\begin{equation}\label{eq:Psi}
 \Upsilon_k(t) = \sup_{0 \leq l \leq k} \sup_{x \in \mathbb R} \bigg[ (1-\gamma)\Theta_{k-l}\left(\frac{t-\gamma x}{1-\gamma}\right) +\gamma \left(\Lambda_l(x)-\frac{x^{2}}{2} \right)\bigg] .
\end{equation}
\end{proof}
In the case of the pure $p,q$ Hamiltonian, one has for instance the expression (see \cite[Equation (2.16)]{ABC})
\begin{equation*}\begin{split}
\Upsilon_0(t)= \sup_{x \in \mathbb{R}}\bigg \{  &(1-\gamma)\bigg( \frac{1}{2} \log (q-1) - \frac{q-2}{4(q-1)} \frac{(t-\gamma x)^2}{(1-\gamma)^2} - I\big (\frac{t-\gamma x}{1-\gamma} \big) \bigg ) \\ &+\gamma \bigg( \frac{1}{2} \log (p-1) - \frac{p-2}{4(p-1)}  x^2 - I(x)\bigg)    \bigg \}
\end{split}
\end{equation*}
where the function $I$ is defined in \eqref{kkakddddfaso}.

\section{Proof of Theorem \ref{upperlower}} \label{sec:s5}

Recall $\mathcal{H}_{N_1}^1$ and $\mathcal{H}_{N_2}^2$ from \eqref{eq:56} and $\xi(x,y)$ from \eqref{caraca}. From now on, we will choose
\begin{align}
\label{eq-12}
\beta_p^2=\sum_{q\geq 1}\beta_{p,q}^2\,\,\mbox{and}\,\,{\beta_q'}^2=\sum_{p\geq 1}\beta_{p,q}^2
\end{align}
for all $p,q\geq 1.$ To lighten notations, we will simply denote $H_N$ as $H$ and $\tilde{H}_N$ as $\tilde{H}$. The hypothesis on $H$ and $\tilde{H}$ allows us to apply Rice's formula in the form of Lemma 3.1 of \cite{ABC}, which says 
  \begin{align}
\begin{split}\label{e:meta} \e \Crt_{N,k}^{H}(t)=  V_{N_1} V_{N_2} \int_{S^{N_1} \times S^{N_2}}
   \zeta(u,v) \d \lambda_{N_1}\times\lambda_{N_2}(u,v),\\
    \e \Crt_{N,k}^{\tilde{H}}(t)=  V_{N_1} V_{N_2} \int_{S^{N_1} \times S^{N_2}}
      \tilde{\zeta}(u,v) \d \lambda_{N_1}\times\lambda_{N_2}(u,v),
      \end{split}
  \end{align}
where 
\begin{equation*}\label{eq:firststepq2}
V_N = \frac{2(\pi N)^{N/2}}{\Gamma (N/2)}
\end{equation*}
and
\begin{align}
\begin{split}\label{def:zeta}
    \zeta(u,v)&:= \e \big[ | \det \nabla^2H(u,v) |
      \indi\{H(u,v) \leq N t,i(\nabla^2 H(u,v))=k\}\, \big|\,
      \nabla H(u,v) = 0 \big] \phi_{(u,v)}(0), \\
      \tilde{\zeta}(u,v)&:= \e \big[ | \det \nabla^2\tilde{H}(u,v) |
            \indi\{\tilde{H}(u,v) \leq N t,i(\nabla^2 \tilde{H}(u,v))=k\}\, \big|\,
            \nabla \tilde{H}(u,v) = 0 \big] \tilde{\phi}_{(u,v)}(0).
            \end{split}
\end{align} 
Here $\phi_{(u,v)}$ and $\tilde{\phi}_{(u,v)}$ are the joint densities of the gradient vectors of $H$ and $\tilde{H}$ at $(u,v)$, respectively. Due to rotation symmetry of the law of the Hamiltonian $H_N$, we have 

\begin{lemma}
For any $u_1, u_2 \in S^{N_1}$ and $ v_1, v_2 \in S^{N_2}$,  we have $\zeta(u_1,v_1) = \zeta(u_2,v_2)$ and $\tilde{\zeta}(u_1,v_1)=\tilde{\zeta}(u_2,v_2).$
\end{lemma}

\begin{proof}
Let $\theta_1$ and $\theta_2$ be orthogonal transformations of $S^{N_1}$ and $S^{N_2}$ that send $u_1$ to $u_2$ and $v_1$ to $v_2$, respectively. Recall the continuous orthonormal basis $\Omega$ from the proof of Theorem \ref{thm:easycomplexity}. Write $\theta E$ for the image of $\Omega$ under the rotation $(\theta_1, \theta_2)$. Then $\hat H(u,v) := H(\theta_1 u, \theta_2 v )$ has the same law as $H$ since they share the same mean and covariance function. Although it is not necessarily true that $\theta_1 E_i^1(\theta_1 u, \theta_2 v) = E_i^1(\theta_1 u, \theta_2 v)$ or  $\theta_2 E_i^2(\theta_1 u, \theta_2 v) = E_i^2(\theta_1 u, \theta_2 v)$, the value of $|\det \nabla^2 \hat H (u,v) | $ is the same regardless of the choice of the basis at the tangent plane.  Furthermore, the law of $\nabla H(u,v)$ is the same as the law of $ \nabla \hat H(u,v)$ (written in the $\theta E$ coordinates) and being a critical point is independent of the choice of the basis at the tangent plane. These facts establish that $\zeta(u_1,v_1)=\zeta(u_2,v_2)$ and same arguments hold true for $\tilde{\zeta}.$
\end{proof}

By this lemma, it suffices to understand the functions $\zeta$ and $\tilde{\zeta}$ at the ''double north pole''$$\boldsymbol n:= (0,\ldots,0,\sqrt{N_1},0,\ldots,0,\sqrt{N_2}).$$ Consequently, since $\lambda_{N_1}\times\lambda_{N_2}$ is a probability measure, we have 
\begin{equation}
\begin{split}\label{eq:firststepq}
 \e \Crt_{N,k}^H(t) &= V_{N_1}V_{N_2} \zeta(\boldsymbol n),\\
  \e \Crt_{N,k}^{\tilde{H}}(t) &= V_{N_1}V_{N_2} \tilde{\zeta}(\boldsymbol n).
 \end{split}
\end{equation}
To proceed we will need one lemma. It establishes the covariance structure of the vector $(H(\boldsymbol n),~ \nabla H(\boldsymbol n),~\nabla^2 H(\boldsymbol n))$. We use the notation given in \eqref{GradientDecomposition} and \eqref{HessianDecomposition}. Since the difference between $H$ and $\tilde H$ will be mainly present on the non-diagonal blocks of their Hessians, we also include the covariance structure for $(\tilde H(\boldsymbol n), \nabla \tilde H(\boldsymbol n), \nabla^2 \tilde H(\boldsymbol n))$ to emphasize to the reader the difference between the coupled Hamiltonian $\tilde H$ and the Hamiltonian of the bipartite model. Recall the function $\xi(x,y)$ from \eqref{eq-1}. Set
\begin{equation*}\label{eq:constants}
\xi'_1 = \partial_x \xi(1,1),  \xi'_2 = \partial_y \xi(1,1), \xi''_1 = \partial_x^2 \xi(1,1),  \xi''_2 = \partial_y^2 \xi(1,1)
\end{equation*}
and for $i=1,2$,
\begin{equation}\label{def:alpha}
\alpha_i = (\xi_i'' + \xi_i' - \xi_i'^2)^{1/2}.
\end{equation}
Here the quantity in the square root of \eqref{def:alpha} is nonnegative as can be easily verified using $\xi(1,1)=1$ and the Cauchy-Schwarz inequality. So $\alpha_i$ is well-defined. Set $I_1=\{1,2,\ldots,N_1-1\}$ and $I_2=\{1,2,\ldots,N_2-1\}.$

\begin{lemma} \label{l:conditioning}
Suppose that $X=H$ or $\tilde{H}$. Write
\begin{equation*}\label{GradientDecompositionX}
 \nabla X(\boldsymbol n) = (\nabla X^1(\boldsymbol n), X^2(\boldsymbol n)) 
 \end{equation*}
and 
\begin{equation*} \label{HessianDecompositionX}
\nabla^2 X(\boldsymbol n) =\left( \begin{array}{cc}
\nabla^2 X^{11}(\boldsymbol n) & \nabla^2 X^{12}(\boldsymbol n)   \\
\nabla^2 X^{21}(\boldsymbol n) &  \nabla^2 X^{22}(\boldsymbol n) \end{array} \right)
\end{equation*}
as in \eqref{GradientDecomposition} and \eqref{HessianDecomposition}. The following statements hold.

  \begin{enumerate}
  
  \item[$1$] The gradient $\nabla X(\boldsymbol n)$ is independent of the vector $(X(\boldsymbol n), \nabla^2 X(\boldsymbol n)).$ Furthermore, $\nabla X(\boldsymbol n)$ is a vector of $N - 2$ independent Gaussian random variables with variance
  \begin{align}
  \label{eq-13}
   \e[\nabla X_i^1(\boldsymbol n)\nabla X_i^1(\boldsymbol n)]  = \frac{\xi'_1N}{N_1} , \,\,\e[\nabla X_j^2(\boldsymbol n) \nabla X_j^2(\boldsymbol n)]  =  \frac{\xi'_2N}{N_2},\,\, \forall (i,j)\in I_1\times I_2.
  \end{align}

  \item[$2$] The only correlated entries of the matrix $\nabla^2 X^{11}(\boldsymbol n)$ are in the diagonal. Precisely,  for $i,i'\in I_1$ with $i\neq i',$
 \begin{equation*}
 \begin{split}
  \Var[ \nabla^2 X^{11}_{ii}(\boldsymbol n)] &= \frac{N}{N_1^2}  (\xi'_1 + 3 \xi''_1 ),\\
  \Var[ \nabla^2 X^{11}_{ii'}(\boldsymbol n)] &= \frac{N}{N_1^2} \xi''_1, \\
  \e[ \nabla^2 X^{11}_{ii}(\boldsymbol n)\nabla^2 X^{11}_{i'i'}(\boldsymbol n)] &=  \frac{N}{N_1^2} (\xi'_1 + \xi''_1).
    \end{split} 
  \end{equation*}
  
  \item[$3$] The covariance structure of the matrix $\nabla^2 X^{22}(\boldsymbol n)$ is as in the item 2 above replacing $\xi'_1, \xi''_1$ by $\xi'_2$, $\xi''_2$ and $N_1$ by $N_2$.

\item[$4$] The diagonal entries of $\nabla^2 X^{11}(\boldsymbol n)$, $\nabla^2 X^{22}(\boldsymbol n)$ are correlated as
$$   \e[ \nabla^2 X^{11}_{ii}(\boldsymbol n)\nabla^2 X^{22}_{jj}(\boldsymbol n)] =  \frac{N}{N_1N_2}\xi'_1\xi'_2,\,\,\forall (i,j)\in I_1\times I_2.$$ All other correlations between entries of $ \nabla^2 X^{11}(\boldsymbol n)$ and $ \nabla^2 X^{22}(\boldsymbol n)$ are zero.
  
  \item[$5$] The  off-diagonal block matrices $ \nabla^2 X^{12}(\boldsymbol n)$ and $ \nabla^2 X^{21}(\boldsymbol n)$ are independent of both $ \nabla^2 X^{11}(\boldsymbol n)$ and $\nabla^2 X^{22}(\boldsymbol n)$. All their entries are independent and with variance given by: $\forall (i,j)\in I_1\times I_2,$
  
   \begin{equation*}
\Var[ \nabla^2 X^{12}_{ij}(\boldsymbol n)] =\Var[ \nabla^2 X^{21}_{ji}(\boldsymbol n)] = \begin{cases} \frac{N}{N_1N_2}  \xi'_1\xi'_2, &\text{ if } X =H, \\
0, &\text{ if } X = \tilde H.
\end{cases} 
\end{equation*}

  \item[$6$] The off-diagonal entries of $\nabla^2X(\boldsymbol n)$ are independent of $X$. The diagonal terms satisfy 
   \begin{equation*}
\e[ \nabla^2 X^{11}_{ii}(\boldsymbol n)X(\boldsymbol n) ] = - \frac{N}{N_1} \xi'_1, \quad \e[ \nabla^2 X^{22}_{jj}(\boldsymbol n)X(\boldsymbol n) ] = - \frac{N}{N_2} \xi'_2,\,\,\forall (i,j)\in I_1\times I_2.
\end{equation*}

\item[$7$] Conditioning on $ X(\boldsymbol n)=x,$ the off-diagonal entries of the random matrix $\nabla^2 X(\boldsymbol n)$ are independent centered Gaussian random variables. Furthermore, the diagonal entries have mean
  \begin{equation*}
    \mathbb E[\nabla^2 X^{11}_{ii}(\boldsymbol n)]=- \frac{\xi'_1x}{N_1}, \quad  \mathbb E[\nabla^2 X^{22}_{jj}(\boldsymbol n)]=- \frac{\xi'_2x}{N_2},\,\,\forall (i,j)\in I_1\times I_2
  \end{equation*}
  and for $i,i'\in I_1$ and $j,j'\in I_2,$
  \begin{equation*}
  \begin{split}
    \e\big[\nabla^2 X^{11}_{ii}(\boldsymbol n) \nabla^2 X^{11}_{i'i'}(\boldsymbol n)\big]&= \frac{N}{N_1^2} (\xi'_1 + (1+2\delta_{ii'}) \xi''_1  - \xi_1'^2), \quad  \mathbb \Var\big[\nabla^2 X^{11}_{ii'}(\boldsymbol n)\big]= \frac{N}{N_1^2}\xi_1'', \\  
    \e\big[\nabla^2 X^{22}_{jj}(\boldsymbol n) \nabla^2 X^{22}_{j'j'}(\boldsymbol n)\big]&= \frac{N}{N_2^2} (\xi'_2 + (1+2\delta_{jj'}) \xi''_2  - \xi_2'^2), \quad  \mathbb \Var\big[\nabla^2 X^{22}_{jj'}(\boldsymbol n)\big]= \frac{N}{N_2^2}\xi_2'', \\ 
      \mathbb \e\big[\nabla^2 X^{11}_{ii}(\boldsymbol n)\nabla^2 X^{22}_{jj}(\boldsymbol n)\big]&=  \frac{N}{N_1N_2}  \xi'_1\xi'_2,\\
         \mathbb \Var\big[\nabla^2 X^{12}_{ii'}(\boldsymbol n)\big]&=\mathbb \Var\big[\nabla^2 X^{21}_{jj'}(\boldsymbol n)
         \big]= \begin{cases} \frac{N}{N_1N_2},  \xi'_1\xi'_2 &\text{ if } X =H, \\
    0, &\text{ if } X = \tilde H.
        \end{cases}
    \end{split}
  \end{equation*}
\end{enumerate}
\end{lemma}

\begin{proof} The proof basically follows from standard computations involving derivatives of smooth Gaussian fields as in  \cite[Lemma 8.5]{Azais}. We will only provide the proof for $X=H$. As for the case $X=\tilde H$, it can be treated in a similar manner by noting that \eqref{eq-12} and the functions $\xi^1(x)=\sum_{p\geq 1}\beta_p^2x^p$ and $\xi^2(x)=\sum_{q\geq 1}{\beta_q'}^2x^q$ appeared in the definition of $\tilde{H}$ satisfy
$$
\xi^i(1)=1,\,\,{\xi^i}'(1)=\xi_i',\,\,{\xi^i}''(1)=\xi_i''
$$
for $i=1,2.$ First, we  define the function 
  $\Psi_k :S^{k} \to \mathbb R^{k-1}$ by 
  $\Psi_k (x_1,\dots,x_k)=(x_1,\dots,x_{k-1})$. The function $\Psi: S^{N_1} \times S^{N_2}\to \mathbb R^{N-2}$ given by 
  $$ \Psi (u,v) = (\Psi_{N_1}u, \Psi_{N_2}v)$$ is a chart in some 
  neighborhood $U$ of $\boldsymbol n$. 
  We set
  \begin{equation*}
    \bar H = H\circ \Psi^{-1},
  \end{equation*}
  which is a Gaussian process on $\Psi (U)$ with covariance
  \begin{equation}\label{eq:covHJ}\begin{split}
    &C((u,v),(w,z))\\
    &=\Cov(\bar H(u,v),\bar H(w,z))\\
    &= \sum_{p,q\geq 1}  \frac{\beta_{p,q}N}{N_1^pN_2^q}
    \bigg\{\sum_{i=1}^{N_1-1} u_iw_i
      +\sqrt{\big(N_1- \textstyle\sum_{i=1}^{N_1-1} u_i^2\big)
        \big(N_1- \sum_{i=1}^{N_1-1} w_i^2\big)}\bigg\}^p \\& 
        \qquad\qquad\qquad \cdot\bigg\{\sum_{i=1}^{N_2-1} v_iz_i
      +\sqrt{\big(N_2- \textstyle\sum_{i=1}^{N-1} v_i^2\big)
        \big(N_2- \sum_{i=1}^{N_2-1} z_i^2\big)}\bigg\}^q.
        \end{split}
  \end{equation}
As in the proof of Theorem \ref{thm:easycomplexity}, we chose the orthogonal basis of the tangent spaces  to coincide with the vector field $E_i, i=1\ldots, N-2$. Then the covariant Hessian $\nabla^2 H(\boldsymbol n)$ coincides with the Euclidean Hessian of
  $\bar H(0,0)$,  by
  noting that the Christoffel symbols $\Gamma^i_{kl}(\boldsymbol n )\equiv 0$.  The covariances of $\bar H, \nabla \bar H, \nabla^2 \bar H$, are computed using a well-known formula \cite[Formula (5.5.5)]{AT07}
$$ \e \bigg\{ \frac{\partial^{\alpha+\beta} \bar H(u,v)}{\partial^\alpha u \partial^\beta v }   \frac{\partial^{\rho+\zeta} \bar H(u,v)}{\partial^\rho u \partial^\zeta v }\bigg\} =  \frac{\partial^{\alpha+\beta+\rho +\zeta} C((u,v), (w,z))}{\partial^\alpha u \partial^\beta v \partial^\rho w \partial^\zeta z },$$  
that relates the covariance of the derivatives of $H$ to the derivatives of the covariance function $C$. Since the derivatives of a centered Gaussian field have centered Gaussian distribution, the covariance matrix determines uniquely the joint distribution of $(H, \nabla H, \nabla^2H)$ and items $1 - 6$ follow from \cite[Formula (5.5.4)]{AT07}
    applied to \eqref{eq:covHJ}.
For instance, item $1$ follows from differentiating $C$  with respect to $u_i$, $w_j$  for all possible values of $i$ and $j$ and evaluating at $u=v=w=z=0$.
 Item $7$ follows from the rules on Gaussian distributions transform under conditioning (see, e.g., \cite[Pages 10-11]{AT07}).
\end{proof}

 Write $I_m$ for the $m \times m$ identity matrix. Item $7$ of the lemma above says that under the law $\mathbb P[\cdot | H(\boldsymbol n)=x]$, the matrix  $\nabla^2H(\boldsymbol n)$ has the same distribution as the random matrix 
 
\begin{align}\label{decomposition1}
 \left( \begin{array}{cc}
G_{1} & G \\
G^t & G_{2}  \\
 \end{array} \right) - \left( \begin{array}{cc}
\frac{\xi_1'x}{N_1}I_{N_1} & 0 \\
0 & \frac{\xi_2' x}{N_2}I_{N_2}   \\
 \end{array} \right).
 \end{align}
Here, $G$ is a $(N_1-1 )\times (N_2-1)$ matrix with independent Gaussian entries with mean zero and variance $N\xi'_1\xi'_2 /(N_1N_2)$ and for $i=1,2$,
\begin{equation}\label{eq:Gi}
G_i \stackrel{d}{=}\left(\frac{N(N_i-1)}{N_i^2}2\xi_i''\right)^{1/2}M^{N_i-1} + \frac{\sqrt{N}}{N_i}\alpha_i Z_i I_{N_i},
\end{equation}
where $\alpha_i$ is defined by \eqref{def:alpha}, $Z_i$'s are independent standard Gaussian random variables and $M^{N_i-1} $ is a $(N_i-1) \times (N_i-1)$ GOE matrix (see \eqref{Goedef} below). To proceed, under the law $\mathbb P[\cdot | H(\boldsymbol n)=x]$, we consider the decomposition $H= H_1 + H_2 $ for $a\in [0,1]$, where $H_1$ and $H_2$ are independent random matrices defined as
\begin{align}
\begin{split}\label{decomposition}
H_1 &= H_1(a)= \left( \begin{array}{cc}
G_{1} & 0 \\
0 & G_{2}  \\
 \end{array} \right) - a\left( \begin{array}{cc}
\frac{\xi_1'x}{N_1}I_{N_1} & 0 \\
0 & \frac{\xi_2' x}{N_2}I_{N_2}   \\
 \end{array} \right),\\
H_2 &= H_2(a)= \left( \begin{array}{cc}
0 & G  \\
G^t & 0  \\
 \end{array} \right) - (1-a)\left( \begin{array}{cc}
\frac{\xi_1'x}{N_1}I_{N_1} & 0 \\
0 & \frac{\xi_2' x}{N_2}I_{N_2}   \\
 \end{array} \right).
 \end{split}
 \end{align}

 \begin{remark}\rm
 This is the part where the computation for $H$ and $\tilde H$ goes in different ways. The reader familiar with the computations performed in \cite{AB} and  \cite{ABC} can foresee that modulo some small technical points the computation for $\tilde H$ follows directly the path taken for the mixed $p$-spin model. The reason is that for $\tilde H$ the matrix $G$ above is identically zero and the Hessian becomes a diagonal block matrix where each block is exactly the Hessian of a mixed $p$-spin model (see the proof of Theorem \ref{thm:easycomplexity}). Therefore, one does not need the matrix $H_2$ in \eqref{decomposition}.
 However, the Hamiltonian $H$ of the bipartite model does not share this property and one needs to deal with the off diagonal block matrix. As far as we know, exact formulas are not available and one cannot perform exact computations. The alternative taken here is to try to bound the (determinant of)  Hessian $H$ by its block terms using the decomposition \eqref{decomposition}, where the constant $a$ plays the weight assigned to each part of the sum. Optimizing over the choice of $a$ will lead to Theorem \ref{upperlower}.
 \end{remark}

Recall the function $\zeta$ defined in \eqref{def:zeta} and consider the case $k=0$. Note that $H(\boldsymbol n)$ is centered Gaussian with variance $N\xi(1,1)=N.$ We first condition on the value of $H$ and use item $1$ of Lemma \ref{l:conditioning} to write

\begin{equation}\label{eq:middlestep}\begin{split}
\zeta(\boldsymbol n) = \frac{\phi_{\boldsymbol n}(0) \sqrt{N}}{\sqrt{2\pi}}  \int_{-\infty}^{t} \e \big[ | \det \nabla^2H(\boldsymbol n) | \indi_{\{ i(\nabla^2 H(\boldsymbol n))=0\}} \, \big|\, H(\boldsymbol n) = N x \big] e^{- \frac{Nx^2}{2}} \d x,\\
\tilde{\zeta}(\boldsymbol n) = \frac{\tilde{\phi}_{\boldsymbol n}(0) \sqrt{N}}{\sqrt{2\pi}}  \int_{-\infty}^{t} \e \big[ | \det \nabla^2\tilde{H}(\boldsymbol n) | \indi_{\{ i(\nabla^2\tilde{ H}(\boldsymbol n))=0\}} \, \big|\, \tilde{H}(\boldsymbol n) = N x \big] e^{- \frac{Nx^2}{2}} \d x,
      \end{split}
\end{equation}
where recalling that $\phi_{\boldsymbol{n}}$ and $\tilde{\phi}_{\boldsymbol{n}}$ are the densities for $H$ and $\tilde{H}$  respectively, item $1$ in Lemma \ref{l:conditioning} implies that they are both centered Gaussian with covariance structures \eqref{eq-13} and thus,
\begin{align}
\begin{split}\label{eq-14}
\phi_{\boldsymbol n}(0) = \tilde\phi_{\boldsymbol n}(0) =\bigg( (2\pi N)^{N-2} \left(\frac{\xi_1'}{N_1}\right)^{N_1-1} \left(\frac{\xi_2'}{N_2}\right)^{N_2-1}\bigg) ^{-1/2}.
\end{split}
\end{align}
The control of $\e \Crt_{N,0}^{H_N}(t)$ from above and below proceeds is discussed in the following two subsections.

\smallskip

\subsection{ Upper bound} Take the decomposition \eqref{decomposition} with $a=1$. Recall that Fischer's inequality (\cite[Fact 8.11.26]{Bern}) says that if $$
\mathcal A = \left( \begin{array}{cc}
A & B  \\
B^t & C  \\
 \end{array} \right)$$ is positive definite, then $\det \mathcal A \leq \det A \det C$. From this, we obtain
 $$| \det \nabla^2H(\boldsymbol n) | \indi_{\{ i(\nabla^2 H(\boldsymbol n))=0\}} \leq \prod_{i=1}^2 | \det \nabla^2H^{ii}(\boldsymbol n) | \indi_{\{ i(\nabla^2 H^{ii}(\boldsymbol n))=0\}}.$$
In view of \eqref{eq:middlestep}, applying item $7$ in Lemma \ref{l:conditioning} for both $H(\boldsymbol{n})$ and $\tilde{H}(\boldsymbol{n})$ and using \eqref{eq-14} lead to
\begin{equation*}\begin{split} 
\zeta(\boldsymbol{n})&\leq   \frac{\phi_{\boldsymbol n}(0) \sqrt{N}}{\sqrt{2\pi}} \int_{-\infty}^t\prod_{i=1}^2\e \left[ \det \left(G_i -\frac{\xi_i'x}{N_i} I_{N_i} \right)\indi_{\left\{ G_i -\frac{\xi_i'x}{N_i} I_{N_i}\geq 0 \right\}} \right]e^{-\frac{x^2}{2N}}dx \\
&=  \frac{\phi_{\boldsymbol n}(0) \sqrt{N}}{\sqrt{2\pi}} \int_{-\infty}^t\e[\det \tilde H_N (\boldsymbol n)  \indi_{\{ \nabla^2 \tilde{H}_N(\boldsymbol{n})\geq 0 \}} \big| \tilde H_N (\boldsymbol n) = Nx ]e^{-\frac{x^2}{2N}}dx\\
&=\tilde{\zeta}(\boldsymbol{n}).
\end{split} 
\end{equation*}
From the definition \eqref{e:meta} and \eqref{def:zeta} of the complexity functions for $H(\boldsymbol{n})$ and $\tilde{H}(\boldsymbol{n})$,
we get $\e \Crt_{N,0}^{H}(t) \leq \e \Crt_{N,0}^{\tilde H}(t)$ for all $N$ and all $t$. From Theorem \ref{thm:easycomplexity}, we obtain the upper bound with $K(t) = \Upsilon_0(t)$. The fact that $\lim_{t \rightarrow -\infty}K(t) = -\infty$ follows from \eqref{eq:Psi} by noting that $\lim_{t\rightarrow -\infty} \Theta_0(t)= -\infty$ and $\Lambda_0$ is bounded above (see \cite[Proposition 1]{AB}).
\smallskip

\subsection{Lower bound} Recall the decomposition $H_1$ and $H_2$ from \eqref{decomposition} by replacing $x$ by $Nx.$ From item $7$ in Lemma \ref{l:conditioning}, under $\Pro[\cdot | H(\boldsymbol{n})=Nx]$, $|\det \nabla^2 H| \indi_{\{ i(\nabla^2 H(\boldsymbol n))=0\}}$ has the same law as $|\det (H_1+H_2)| \indi_{\{ i(\nabla^2 H(\boldsymbol n))=0\}} $. On the other hand,
\begin{equation*}
\begin{split}
|\det (H_1+H_2)| \indi_{\{ i(\nabla^2 H(\boldsymbol n))=0\}}  &\geq  |\det (H_1+H_2)| \indi_{\{ H_1 \geq 0 \}}\indi_{\{ H_2 \geq 0 \}}\\
&\geq (\det H_1 + \det H_2)\indi_{\{ H_1\geq 0 \}}\indi_{\{ H_2\geq 0 \}},
\end{split}
\end{equation*}
where the second inequality used the Minkowski determinant inequality \cite[Corollary 8.4.15]{Bern}.
This implies that 
\begin{equation}\label{eq:mainlower}
\begin{split}
 &\e \big[ | \det \nabla^2H(\boldsymbol n) | \indi_{\{ i(\nabla^2 H(\boldsymbol n))=0\}}\big|\,H(\boldsymbol n) = N x \big] \\
 &\geq \sup_{a \in [0,1]} \e \big[ \det H_1 \indi_{\{ H_1 \geq 0 \} }  \big]\Pro (H_2 \geq 0) \\
 &=\sup_{a \in [0,1]} \left\{\Pro (H_2 \geq 0)\prod_{i=1}^2\e \left[ \det \left(G_i -a\frac{N\xi_i'x}{N_i} I_{N_i} \right)\indi_{\left\{ G_i -a\frac{N\xi_i'x}{N_i} I_{N_i}\geq 0 \right\}} \right] \right\}.
 \end{split}
\end{equation} 
 
 \begin{lemma} \label{lem5} Let $\gamma_*= \max\{\gamma(1-\gamma)^{-1}, \gamma^{-1}(1-\gamma)\}$. If \begin{equation}\label{lem:condition}(1-a)\min\left\{ \frac{\xi_1'x}{\gamma},  \frac{\xi_2' x}{(1-\gamma)}\right\} < - (1+ \sqrt{\gamma_*}),\end{equation} then
$  \lim_{N \rightarrow \infty} \Pro (H_2 \geq 0) = 1.$
\end{lemma}

 \begin{proof} Denote by $\lambda_m(A)$ and $\lambda_M(A)$ the minimum and maximum eigenvalues for any symmetry matrix $A$. Note that $H_2$ is positive definite if and only if $\lambda_m(H_2)>0$. Writing $\lambda_{m}(H_2) = \min_{v\neq 0} \< H_2 v, v \>/ \< v, v\>$, this expression implies that $\lambda_{m}(H_2)>0$ if $$\lambda_m(\mathcal{G})> c_N:=(1-a)\min\left\{ \frac{N\xi_1'x}{N_1},  \frac{N\xi_2' x}{N_2}\right\}$$ 
 for $$ \mathcal G:= \left( \begin{array}{cc}
0 & G  \\
G^t & 0  \\
 \end{array}\right).$$
Let $$
c=\lim_{N\rightarrow\infty}c_N=(1-a)\min\left\{ \frac{\xi_1'x}{\gamma},  \frac{\xi_2' x}{(1-\gamma)}\right\}.
$$
It is a well-known fact (see \cite[Exercise 2.1.18]{AGZbook}) that the largest eigenvalue of $GG^t$ converges almost surely to the edge of the Marchenko-Pastur law, which in our normalization is given by $(1 + \sqrt{\gamma_*})^2$. Thus, with probability one,  $\limsup_{N\rightarrow\infty}\lambda_M(GG^t)\leq (1+\sqrt{\gamma^*})^2+\varepsilon$ for any $\varepsilon>0$. Using this and the fact that $\lambda^2$ is an eigenvalue of $GG^t$ for any eigenvalue $\lambda$ of $\mathcal{G}$, we obtain that if $c<-(1+\sqrt{\gamma^*}),$ then 
\begin{align*}
c^2+\varepsilon>(1+\sqrt{\gamma^*})^2+\varepsilon\geq\limsup_{N\rightarrow\infty}\lambda_M(GG^t)\geq \limsup_{N\rightarrow\infty}\lambda_{m}(\mathcal{G})^2
\end{align*}
and letting $\varepsilon\downarrow 0$ gives $c^2>\limsup_{N\rightarrow\infty}\lambda_{m}(\mathcal{G})^2.$
Since $G$ is fromed by i.i.d. standard Gaussians, $\lambda_m(\mathcal{G})$ is almost surely negative. These amount to say that if $c<-(1+\sqrt{\gamma^*}),$ 
$
c< \liminf_{N\rightarrow\infty}\lambda_{m}(\mathcal{G})
$
with probability one.
Using Fatou's lemma, we are done since
\begin{align*}
1&=\p\left(c< \liminf_{N\rightarrow\infty}\lambda_{m}(\mathcal{G})\right)\\
&\leq \int\liminf_{N\rightarrow\infty}1\left\{\lambda_{m}(\mathcal{G})>c_N\right\}d\p\\
&\leq \liminf_{N\rightarrow\infty}\int 1\left\{\lambda_{m}(\mathcal{G})>c_N\right\}d\p\\
&\leq \liminf_{N\rightarrow\infty}\p(H_2\geq 0).
\end{align*}
\end{proof}
This lemma implies that if we choose $x$ and $a$ such that \eqref{lem:condition} holds, then for $N$ large enough $\Pro(H_2 \geq 0)$ is bounded below by a constant $C>0$. Now we need to investigate the term 
$$\prod_{i=1}^2\e \left[ \det \left(G_i -a\frac{N\xi_i'x}{N_i} I_{N_i}\right)\indi_{\left\{ G_i -a\frac{N\xi_1'x}{N_i} I_{N_i}\geq 0 \right\}} \right]$$
for which we will need Lemma 3.3 \cite{ABC}. Its precise statement is given below. Recall that a real symmetry matrix $M^N$ of size $N\times N$ is a Gaussian Orthogonal Ensemble (GOE) if the entries $(M_{ij}^N, i\leq j)$ are 
independent Gaussian random variables with mean zero and variance
\begin{equation}\label{Goedef}
  \mathbb E(M_{ij}^N)^2 = \frac{1+\delta_{ij}}{2N}.
\end{equation}
Denote by  $\e^N_{G}$ the expectation under the GOE of
size $N\times N$.

\begin{lemma}$\cite[\mbox{Lemma}\,\,3.3]{ABC}$
\label{LemmafromGA} Let $M^{N-1}$ be a $(N-1)\times (N-1)$ GOE matrix, $\lambda_k^{(N-1)}$ be the $k$-th smallest eigenvalue of  $M^{N-1}$ and $Y$ be
  an independent Gaussian random variable with mean $m$ and variance $t^2$. Then, for any Borel set $B\subset\mathbb R$,
 \begin{equation}
\begin{split}  \label{eq-15}\e &\Bigg[ \big|\det(M^{N-1}- Y I)\big| \indi\big\{ i(M^{N-1}- Y I)=k, Y \in B\big\} \Bigg]
  \\&= \frac{\Gamma(\frac{N}{2})(N-1)^{-\frac {N}{2}}}{\sqrt{\pi t^2}} \e_G^N \bigg[ \exp\bigg\{\frac{N(\lambda_{k}^{N})^2}{2} - \frac{\big(\left(\frac{N}{N-1}\right)^{\frac {1}{2}}\lambda_{k}^{N}-m\big)^2}{2t^2}\bigg\} \indi\Big\{ \lambda_{k}^{N} \in \left(\tfrac{N-1}{N}\right)^{\frac12}B\Big\} \bigg]. 
        \end{split}
\end{equation}
\end{lemma}

\smallskip

Assume first that $\alpha_i \neq 0$ for $i =1,2$. In the notation of Lemma \ref{LemmafromGA}, we take $B=\mathbb R$, $N=N_i$ and $Y$ to be a Gaussian random variable of mean $(N (2\xi_i''(N_i-1))^{-1})^{1/2}a\xi_i'x$ and variance $\alpha_i^2 (2\xi_i''(N_i-1))^{-1}$. Recalling that $N_1/N =\gamma$, \eqref{eq-15} together with \eqref{eq:Gi}  gives us the following lower bound for the right side of \eqref{eq:mainlower},
\begin{equation*}
\begin{split}
\frac{1}{2\pi} \prod_{i=1}^2 \frac{\Gamma(\frac{N_i}{2})  (2\xi_i'')^{\frac{N_i+1}{2}}} {\alpha_i} \bigg (\frac{N}{N_i^2}\bigg)^{\frac{N_1-1}{2}}\e_G^{N_i} \bigg[ \exp\bigg \{ \frac{N_i}{2} \bigg((\lambda_0^{N_i})^2 - \frac{2 \xi_i''}{\alpha_i^2} \bigg(\lambda_0^{N_i} - \frac{a\xi_i' x}{\gamma_i(2\xi_i'')^{\frac{1}{2}}}\bigg)^2\bigg)  \bigg \} \bigg],
 \end{split}
 \end{equation*}
 where $\gamma_1 := \gamma$, $\gamma_2:=1-\gamma$ and $\lambda_0^{N_1}$ and $\lambda_0^{N_2}$ are the smallest eigenvalues of two independent GOE of sizes $N_1$ and $N_2$, respectively. Now combining the equation above with \eqref{eq:firststepq}, \eqref{eq-14}, \eqref{eq:mainlower} and Lemma \ref{lem5}, a straightforward computation leads to the following asymptotics
 \begin{equation}\label{eq:rio}
\lim_{N\to \infty}\frac{1}{N}\log \e \Crt_{N,0}(t) \geq \frac{\gamma}{2} \log\bigg( \frac{\xi_1''}{\xi_1'} \bigg) +  \frac{1-\gamma}{2} \log \bigg( \frac{\xi_2''}{\xi_2'} \bigg) + \sup_{a\in[0,1]} \lim_{N\to \infty}\frac{1}{N}\log \mathcal I_N(t,a),
\end{equation} 
where 
\begin{equation*}
\mathcal I_N(t,a)=\int_{-\infty}^{\min \{t, a^*\}}e^{-\frac{N x^2}{2}} \prod_{i=1}^2 \e_G^{N_i} \bigg[ \exp\bigg \{ \frac{N_i}{2}\bigg((\lambda_0^i)^2 - \frac{2 \xi_i''}{\alpha_i^2} \bigg(\lambda_0^i - \frac{a\xi_i' x}{\gamma_i(2\xi_i'')^{\frac{1}{2}}}\bigg)^2 \bigg)  \bigg \} \bigg] dx
\end{equation*}
with $$a^*= - \frac{1+ \sqrt{\gamma_*}}{(1-a)\min \{ \frac{\xi_1'}{\gamma}, \frac{\xi_2'}{1-\gamma}\}}.$$

The asymptotics of $\mathcal I_N(t,a)$ for $a$ and $t$ fixed and $N$ going to infinity can be derived by an application of Laplace-Varadhan's lemma \cite[Theorem 4.5.1]{DemboZeitouni}.  Indeed, the smallest eigenvalue of the GOE matrix satisfies a LDP with speed $N$ and rate function 
\begin{equation}\label{kkakddddfaso}
I(x) =  \int_{\sqrt{2}}^{|x|} \sqrt{z^2-2}\;  d z = \frac{1}{2} \bigg(|x| \sqrt{x^2-2}+\log 2 -2 \log  \left(|x|+\sqrt{x^2-2}\right) \bigg)
\end{equation}
if $x\leq-\sqrt{2}$ and $I(x) = \infty$ for $x>- \sqrt{2}$ (see \cite[Theorem 6.2]{BenArousGuiDembo}). 
Therefore,
\begin{equation*}
\lim_{N\to \infty} \frac{1}{N} \log \mathcal I_N(t,a) = \sup_{ \stackrel{x\leq \min\{ t, a^*\}}{y_1,y_2 \leq -\sqrt{2}} } \bigg[ -\frac{x^2}{2} + \sum_{i=1}^2 \gamma_i \bigg(\frac {y_i^2}{2} - \frac{\xi_i''}{\alpha_i^2} \bigg(y_i - \frac{a\xi_i' x}{\gamma_i(2\xi_i'')^{\frac{1}{2}}}\bigg)^2- I(y_i)\bigg) \bigg].
\end{equation*}
Denoting the right-hand side of \eqref{eq:rio} as $J(t)$ ends the proof of the lower bound in the case $\alpha_1 ,\alpha_2\neq 0$. To see that $J(t)$ is positive for some $t$, it suffices to take $y_1=y_2=-\sqrt{2}$ and $t\geq a^*=x$ in the equation above. This computation is similar to \cite[Equation (2.15)]{AB}. The case when one (or both) $\alpha_i = 0$ is simpler as one (or both) of the parties becomes simply a pure $p$-spin model and one just needs to apply Lemma \ref{LemmafromGA} with $Y=H_{N}$ as in \cite[Equation (3.25)]{ABC}. We leave the details to the reader.

\end{document}

%% file: invlat.tex
\font\tencmmib=cmmib10 \skewchar\tencmmib '60
\newfam\cmmibfam
\textfont\cmmibfam=\tencmmib

\def\lessim{\ \lower4pt\hbox{$
\buildrel{\displaystyle <}\over\sim$}\ }
\def\gessim{\ \lower4pt\hbox{$\buildrel{\displaystyle >}
\over\sim$}\ }

\def\go0{\to 0}

\def\leftitem#1{\item{\hbox to\parindent{\enspace#1\hfill}}}

\def\sg{\sigma}

\def\sg2{\sigma^2}

\def\__{_{\infty}}